\documentclass[11pt,a4paper]{amsart}

\usepackage{amssymb, amsthm, amsmath, pifont}
\usepackage{mathabx} 
\usepackage[usenames]{color}
\usepackage{amsfonts}
\usepackage{enumerate}
\usepackage{verbatim}
\usepackage[latin1]{inputenc}
\usepackage{mathrsfs}
\usepackage{hyperref}
\usepackage[all]{xy}

\newtheorem{theorem}{Theorem}
\newtheorem*{theorem*}{Theorem}
\newtheorem*{maintheorem*}{Main Theorem}

\newtheorem{lemma}[theorem]{Lemma}
\newtheorem{example}[theorem]{Example}
\newtheorem{definition}[theorem]{Definition}
\newtheorem{remark}[theorem]{Remark}

\newtheorem{prop}[theorem]{Proposition}
\newtheorem{claim}[theorem]{Claim}

\newcommand{\tr}[1]{{\color{black} #1}}

\DeclareMathOperator{\GL}{GL}
\DeclareMathOperator{\SL}{SL}
\DeclareMathOperator{\Aut}{Aut}
\DeclareMathOperator{\M}{M}

\DeclareMathOperator{\rank}{rank}

\DeclareMathOperator{\Lie}{Lie}

\def\CC{\mathbb{C}}

\title{Algebraic Embeddings of $\CC$ into $\SL_n(\CC)$}
\author{Immanuel Stampfli}
\address{Jacobs University Bremen gGmbH, School of Engineering and 
Science, Department of Mathematics, Campus Ring 1, 28759 Bremen, Germany}
\email{immanuel.e.stampfli@gmail.com}

\thanks{The author gratefully acknowledge support by the Swiss National
Science Foundation Grant ``Automorphisms of Affine Algebraic Varieties" 148627.}

\begin{document}

\begin{abstract}
	We prove that any two algebraic embeddings
	$\CC \to \SL_n(\CC)$ are the same up to an algebraic automorphism of 
	$\SL_n(\CC)$,
	provided that $n$ is at least $3$. Moreover, we prove that 
	two algebraic embeddings $\CC \to \SL_2(\CC)$ are the same up
	to a holomorphic automorphism of $\SL_2(\CC)$.
\end{abstract}

\maketitle

\section{Introduction}

There are many results concerning algebraic embeddings of some variety
into the affine space $\CC^n$. Let me recall two of them.
Any two algebraic embeddings of a smooth \tr{affine}
variety $X$ into $\CC^n$ are the same up to an algebraic automorphism of $\CC^n$,
provided that $n > 2 \dim X+1$. This result is due to 
Nori, Srinivas \cite{Sr1991On-the-embedding-d}
and Kaliman \cite{Ka1991Extensions-of-isom}. If one relaxes
the condition that the automorphism of $\CC^n$ must be algebraic,
Kaliman \cite{Ka2013Analytic-extension} and independently, Feller and the author 
\cite{FeSt2014Holomorphically-Eq} proved the following improvement:
Any two algebraic embeddings of a smooth \tr{affine}
variety $X$ into $\CC^n$ 
are the same up to a \emph{holomorphic} automorphism of $\CC^n$,
provided that $n > 2 \dim X$.

As a further development of these results, we study algebraic embeddings of 
$\CC$ into $\SL_n$. This article can be seen as a first example
to understand algebraic embeddings \tr{of a curve} into an arbitrary 
affine algebraic variety with a large automorphism group.

\tr{In dimension zero, Arzhantsev, Flenner, Kaliman, Kutzschebauch and Zaidenberg 
proved that 
two embeddings of a finite set into any irreducible smooth 
affine flexible variety $Z$ are the same 
up to an algebraic automorphism of $Z$, provided that $\dim Z > 1$
\cite{ArFlKa2013Flexible-varieties}. 
Our main result is based on this work.}

\begin{maintheorem*}[see Theorem~\ref{thm.SLn} and Theorem~\ref{thm.SL2}]
	Let $f, g \colon \CC \to \SL_n$ be algebraic embeddings.
	If $n \geq 3$, then $f$ and $g$ are the same
	up to an algebraic automorphism of $\SL_n$
	and if $n = 2$, then $f$ and $g$ 
	are the same up to a holomorphic automorphism
	of $\SL_n$. 
\end{maintheorem*}

To the author's knowledge it is not known, whether all algebraic 
embeddings $\CC \to \SL_2$ are the same up to an algebraic automorphism
of $\SL_2$. Also for algebraic embeddings $\CC \to \CC^3$
it is an open problem, whether all these embeddings are the same up to 
an algebraic automorphism of $\CC^3$, 
see \cite{Sh1992Polynomial-represe} for potential
examples that are not equivalent to linear embeddings.

In fact, in a certain sense the class of algebraic embeddings $\CC \to \SL_2$ 
is as big as the class of algebraic \tr{embeddings} $\CC \to \CC^3$.
More precisely, the following holds. \tr{If 
$g \colon \CC \to \CC^3$, $t \mapsto (g_1(t), g_2(t), g_3(t))$
is an algebraic embedding,} 
then one can apply a (tame) algebraic automorphism of $\CC^3$
such that  afterwards the polynomial $g_2$ divides $g_1 g_3 - 1$ and thus the
following map is an algebraic embedding
\[
	\CC \longrightarrow \SL_2 \, , \quad t \longmapsto
	\begin{pmatrix}
		g_1(t) & (g_1(t) g_3(t) - 1) / g_2(t) \\
		g_2(t) & g_3(t) 
	\end{pmatrix}	\, .
\]   
\tr{The construction of the claimed (tame) algebraic 
automorphism of $\CC^3$ can be 
seen as follows. First one can 
apply a map of the form $(x, y, z) \mapsto (x, y + \lambda, z)$ 
such that afterwards the polynomial
$g_2$ has only finitely many simple roots, say $t_1, \ldots, t_n$.
Now, it is enough to apply some (tame) algebraic automorphism
of the form $(x, y, z) \mapsto (\varphi_1(x, z), y, \varphi_3(x, z))$, which sends the
points $g(t_1), \ldots, g(t_n)$ to the curve $\{ xz = 1, y = 0 \} \subseteq \CC^3$,
see \cite[Lemma~5.5]{KaZa1999Affine-modificatio}.}

\tr{The proof of the main theorem gives a method to construct
the claimed automorphism. However, the proof does not produce 
a computer algorithm that would give such an automorphism. This is 
because the construction in the proof depends on certain zero sets of polynomials.} 

\section{\texorpdfstring{Algebraic automorphisms of $\SL_n$}
{Algebraic automorphisms of SLn}}

\tr{Let us introduce first some notation.
For $i, j$ in $\{1, \ldots, n \}$, we denote the $ij$-th entry of a matrix $X \in \SL_n$ 
by $X_{ij}$. The projection $\SL_n \to \CC$, $X \to X_{ij}$
we denote by $x_{ij}$.}

In the first lemma, we list algebraic automorphisms of $\SL_n$
that we use constantly. The proof is straight forward.

\begin{lemma}
	\label{lem.autos1}
	Let $n \geq 2$ and let $i \neq j$ be integers in $\{ 1, \ldots, n \}$.
	Then, for every polynomial $p$ in the \tr{functions} $x_{kl}$, $k \neq i$,
	the map
	\[
		\SL_n \longrightarrow \SL_n \, , \quad
		X \longmapsto E_{ij}(p(X)) \cdot X
	\]
	is an automorphism, where \tr{$E_{ij}(a)$}
	denotes the elementary matrix with $ij$-th entry equal to \tr{$a$}. 
	Similarly, for every polynomial $q$ in the \tr{functions} $x_{kl}$, $l \neq j$,
	the map
	\[
		\SL_n \longrightarrow \SL_n \, , \quad
		X \longmapsto X \cdot E_{ij}(q(X))
	\]
	is an automorphism.
\end{lemma}

Recall that the group of tame automorphisms of $\CC^n$ is the subgroup
of the automorphisms of $\CC^n$ 
generated by the affine linear maps and the
elementary automorphisms, i.e. the automorphisms of the form
\tr{\[
	(x_1, \ldots, x_n) \longmapsto (x_1, \ldots, x_i + 
	h_i(x_1, \ldots, \widehat{x_i}, \ldots, x_n), \ldots, x_n) \, , 
\]
where $h_i$ is a polynomial not depending on $x_i$.}
In the next result we list automorphisms of $\CC^n$
that \tr{can be lifted to} automorphisms of $\SL_n$ via 
the projection to the first column $\pi_1 \colon \SL_n \to \CC^n$, 
\tr{i.e.
automorphisms $\psi$ of $\CC^n$ such that there exists an
automorphism $\Psi$ of $\SL_n$ (depending on $\psi$) 
that makes the following diagram commutative:
\[
	\xymatrix{
		\SL_n \ar[d]_-{\pi_1} \ar[r]^-{\Psi} & \SL_n \ar[d]^-{\pi_1} \\
		\CC^n \ar[r]^-{\psi} & \CC^n \ar@{}[r]|<{\ \textrm{\normalsize.}} &
	}
\]}

\begin{lemma}
	\label{lem.autos2}
	Let $n \geq 2$. Every tame automorphism of $\CC^n$
	that preserves the origin \tr{can be lifted to} some automorphism
	of $\SL_n$ via $\pi_1 \colon \SL_n \to \CC^n$.
%
%
\end{lemma}

\begin{proof}
	First, remark that the group of tame automorphisms of $\CC^n$
	that preserve the origin is generated by the linear group $\GL_n$
	and by the elementary automorphisms that preserve the origin.
	For every $A \in \GL_n$, \tr{the linear map $x \mapsto A \cdot x$ of $\CC^n$
	can be lifted to the automorphism
	\[
		\SL_n \longrightarrow \SL_n \, , \quad
		X \longmapsto A \cdot X \cdot 
		\text{diag}\left(1, \ldots, 1, (\det A)^{-1} \right)
	\]
	via $\pi_1$}, where 
	$\text{diag}(\lambda_1, \ldots, \lambda_n)$
	denotes the $n \times n$-diagonal matrix with entries 
	$\lambda_1, \ldots, \lambda_n$.
	\tr{Let $\psi$ be an elementary automorphism of $\CC^n$
	that preserves the origin, i.e. there exist $i \in \{ 1, \ldots n \}$
	and polynomials $p_1, \ldots, \widehat{p_i}, \ldots, p_n$
	in the variables $x_1, \ldots, \widehat{x_i}, \ldots, x_n$ such that
	\[
		\psi(x_1, \ldots, x_n) = (x_1, \ldots, x_i + 
		\sum_{j \neq i} x_j p_j(x_1, \ldots, \widehat{x_i}, \ldots, x_n), \ldots, x_n)
		\, . 
	\]
	The automorphism $\psi$ can be lifted to some automorphism
	of $\SL_n$, e.g. to the automorphism
	\[
		\SL_n \longrightarrow \SL_n \, , \quad 
		X \longmapsto \left( \prod_{j \neq i} E_{ij}(p_j(X_{11}, \ldots,
		\widehat{X_{i1}}, \ldots, X_{n1})) \right)
		\cdot X \, ,
	\]
	cf. also Lemma~\ref{lem.autos1}.
	This finishes the proof.}
\end{proof}

\section{A generic projection result}

Let $\tr{V}$ be an algebraic variety. We say that a statement is true for 
\emph{generic} $\tr{v \in V}$ if there exists 
a Zariski dense open subset $\tr{U \subseteq V}$ 
such that the statement is true for all $\tr{v \in U}$.

\begin{lemma}
	\label{lem.GenericProjection}
	Let $n \geq 3$. If 
	$f \colon \CC \to \SL_n$ is an algebraic embedding such
	that the matrices $f(0)-f(1)$ and $f'(0)$ have maximal rank, 
	then, for generic $A \in \tr{\M_{n,n-1}}$ 
	the map
	\[
		\CC \stackrel{f}{\longrightarrow} \SL_n 
		\stackrel{\pi_A}{\longrightarrow} \tr{\M_{n,n-1}}
	\]
	is an algebraic embedding, where
	$\tr{\M_{n,n-1}}$ denotes the space of $n \times (n-1)$-matrices and
	$\pi_A$ is given by $X \mapsto X \cdot A$.
\end{lemma}

\begin{proof}
	Let $\Delta \subseteq \CC^2$ be the diagonal. Consider the following
	(Zariski) locally closed subsets of $\CC^2 \setminus \Delta$:
	\[
		C_i = \{ \, (t, r) \in \CC^2 \setminus \Delta \ | \ \rank(f(t)-f(r)) = i \, \} \, .
	\]
	Consider for every $A \in \tr{\M_{n,n-1}}$ the composition
	\begin{equation}
		\label{eq.Composition}
		\tag{\tr{$\ast$}}
		\xymatrix{
		C_i \ar[r] & \CC^{2} \setminus \Delta
		\ar[rrr]^-{\tr{(t, r)} \mapsto f(t)-f(r)} &&& 
		\tr{\M_{n,n}} \ar[r]^-{\pi_A} & \tr{\M_{n,n-1}} \, .
		}
	\end{equation}
	This map is never zero for generic $A \in \tr{\M_{n,n-1}}$; indeed:
	\begin{itemize}
		\item If $1 < i \leq n$, then \eqref{eq.Composition} is never zero
		provided that $A \in \tr{\M_{n,n-1}}$ has maximal rank.
				
		\item If $i = 1$, then $\dim C_1 \leq 1$, since
		$\dim C_n = 2$ (note that $f(0)-f(1)$ has maximal rank). 
		For $(t, r) \in C_1$, let
		$Z_{(t, r)} = \ker( f(t)-f(r))$.
		Since $\dim C_1 \leq 1 < n-1$, a generic $(n-1)$-dimensional subspace
		\tr{of $\CC^n$}
		is different from $Z_{(t, r)}$ for all $(t, r) \in C_1$. 
		Thus, for generic $A$ the composition
		$\eqref{eq.Composition}$ is never zero.
	\end{itemize}
	Clearly, $C_0 = \varnothing$.
	Hence, we proved that the composition $\pi_A \circ f$ is injective for generic
	$A \in \tr{\M_{n,n-1}}$. 
	Clearly, $\pi_A \circ f$
	is proper for generic $A \in \tr{\M_{n,n-1}}$.
	
	For the immersivity, we have to show for
	generic $A \in \tr{\M_{n,n-1}}$ that
	\begin{equation}
		\label{eq2}
		f'(t) \cdot A \neq 0
	\end{equation}
	for all $t \in \CC$.
	Since $\rank f'(0) = n$, \tr{the set $U = \{ \, t \in \CC \ | \ \rank f'(t) = n \, \}$ 
	is Zariski dense and open in $\CC$.}
	Thus \eqref{eq2} is satisfied for all $A \neq 0$ and for all $t \in U$.
	Since $f$ is immersive, we have $f'(t) \neq 0$ for all $t \in \CC$.
	This implies that for generic $A$ we have $f'(t) \cdot A \neq 0$ 
	for all $t \in \CC$.
\end{proof}

\section{\texorpdfstring{Algebraic embeddings of $\CC$ into $\SL_n$ for $n \geq 3$}
{Algebraic embeddings of C into SLn for n >= 3}}

\begin{theorem}	
	\label{thm.SLn}
	For $n \geq 3$, any two 
	algebraic embeddings of $\CC$ into $\SL_n$ are the same
	up to an algebraic automorphism of $\SL_n$.
\end{theorem} 

\begin{lemma}
	\label{lem:one}
	Let $n \geq 2$.
	Assume that $f \colon \CC \to \SL_n$ is an algebraic embedding such that
	\[
		\CC \stackrel{f}{\longrightarrow} \SL_n 
		\stackrel{\pi_{n-1}}{\longrightarrow}
		\tr{\M_{n,n-1}}
	\]
	is an algebraic embedding, where $\pi_{n-1}$ denotes
	the projection to the first $n-1$ columns.
	Then there exists an algebraic automorphism $\varphi$ of $\SL_n$ such that
	\[
		\CC \stackrel{f}{\longrightarrow} \SL_n \stackrel{\varphi}{\longrightarrow}
		\SL_n \stackrel{\pi_1}{\longrightarrow} \CC^n
	\]
	is  given by $t \mapsto (1, 0, \ldots, 0, t)^T$.
\end{lemma}

\begin{proof}[Proof of Lemma~\ref{lem:one}]
	Assume that $n = 2$.
	Since two algebraic embeddings of $\CC$ into $\CC^2$
	are the same up to an algebraic automorphism of $\CC^2$ 
	(Abhyankar-Moh-Suzuki Theorem, see 
	\cite{AbMo1975Embeddings-of-the-, Su1974Proprietes-topolog}), 
	one can see that there exists 
	an algebraic automorphism of $\CC^2$ that preserves the origin and
	changes the embedding $\pi_1 \circ f \colon \CC \to \CC^2$
	to the embedding $\CC \to \CC^2$, $t \mapsto (1, t)$. Using the fact that
	every algebraic automorphism of $\CC^2$ is tame 
	(Jung'\tr{s} Theorem, see \cite{Ju1942Uber-ganze-biratio}),
	it follows from Lemma~\ref{lem.autos2} that there exists an algebraic
	automorphism 
	$\varphi$ of $\SL_2$ such that $\pi_1 \circ \varphi \circ f(t) = (1, t)$.
	
	Assume that $n \geq 3$.	
	Let $A(t) = \pi_{n-1} \circ f(t)$. 
	\tr{Since  the kernel of $A(t)^T$ is one-dimensional for all $t$,
	the following affine variety
	\[
		E = \{ \, (v, t) \in \CC^n \times \CC \ | \ A(t)^{T} \cdot v = 0 \, \}
	\]
	defines the total space of a line bundle over $\CC$ 
	with projection map $(v, t) \mapsto t$.
	Since $n \geq 3 > \dim E$, this implies that
	there exists a vector $v \in \CC^n$ such that $v^T \cdot A(t)$ is non-zero for 
	all $t \in \CC$. Now, complete $v^T$ to a matrix $B \in \SL_n$
	with last row equal to $v^T$. Since $n \geq 3$,
	there exists a permutation matrix $P \in \SL_n$, with first column
	equal to $(0, \ldots, 0, 1)^T$.
	After applying the automorphism \tr{$X \mapsto B \cdot X \cdot P$}
	of $\SL_n$}, we can assume that
	\begin{enumerate}[i)]
		\item the map $\CC \to \tr{\M_{n,n-1}}$ given by 
			$t \mapsto (f_{ij}(t))_{1 \leq i \leq n, 2 \leq j \leq n}$
			is an algebraic embedding and
		\item the vector $(f_{n2}(t), f_{n3}(t), \ldots, f_{nn}(t))$ is non-zero
			for all $t \in \CC$,
	\end{enumerate}
	where $f_{ij}(t)$ denotes the $ij$-th entry of the matrix $f(t)$.
	By ii), there exist polynomials $\tilde{p}_k \in \CC[t]$, 
	$2 \leq k \leq n$ such that
	\[
		\sum_{k = 2}^{n} f_{nk}(t) \tilde{p}_{k}(t) 
		= t - f_{n1}(t) \, .
	\]
	By i), there exist polynomials $p_k$ in the \tr{functions}
	$x_{ij}$ with $1 \leq i \leq n$, $2 \leq j \leq n$ such that 
	$\tilde{p}_k(t) = p_k(\ldots, f_{ij}(t), \ldots)$.
	Let $\varphi \colon \SL_n \to \SL_n$ be the automorphism 
	\[
		X \longmapsto X \cdot \begin{pmatrix}
						1 & & & \\
						p_2(X) & 1 & & \\
						\vdots & & \ddots & \\
						p_n(X) & & & 1 
					    \end{pmatrix} \, .
	\]
	Clearly, the left down corner of the matrix $\varphi \circ f(t)$
	\tr{is equal to} $t$.
	Now, one can \tr{construct with the aid of Lemma~\ref{lem.autos2}} 
	an automorphism $\psi$ of $\SL_n$ 
	such that the first column of $\psi \circ \varphi \circ f(t)$
	is equal to $(1, 0, \ldots, 0, t)^T$. This proves the lemma.
\end{proof}

\begin{lemma}
	\label{lem:two}
	Let $n \geq 2$ 
	and let $f \colon \CC \to \SL_n$ be an algebraic embedding
	such that \tr{the first column of $f(t)$ is equal to} $(1, 0, \ldots, 0, t)^T$. 
	Then $f$ is the same as
	\[
		\CC \longrightarrow \SL_n \, , 
		\quad t \longmapsto E_{n1}(t)
	\]
	up to an algebraic automorphism of $\SL_n$,
	where $E_{n1}(t)$ denotes \tr{the} elementary matrix with left down corner
	equal to $t$.
\end{lemma}

\begin{proof}[Proof of Lemma~\ref{lem:two}]
	Let $\psi$ be the automorphism of $\SL_n$ defined by 
	\[
		X \longmapsto X \cdot f(X_{n1})^{-1} 
		\cdot E_{n1}(X_{n1})
	\] 
	\tr{where $X_{ij}$ denotes the $ij$-th entry of the matrix $X$.}
	Now, one can easily check that $\psi \circ f$ is the embedding 
	$t \mapsto E_{n1}(t)$.
\end{proof}

\begin{proof}[Proof of Theorem~\ref{thm.SLn}]
Start with an algebraic embedding $f \colon \CC \to \SL_n$. 
As $\SL_n$ is flexible, for any finite set $Z$ in $\SL_n$ 
there exists an automorphism of $\SL_n$ which fixes $Z$ and has
prescribed \tr{volume preserving differentials} in the points of $Z$, see 
\cite[Theorem 4.14 and Remark 4.16]{ArFlKa2013Flexible-varieties}.
Using the fact that $\Aut(\SL_n)$ acts $2$-transitively on $\SL_n$,
see e.g. \cite[Theorem 0.1]{ArFlKa2013Flexible-varieties},
we can assume that
\[
	\det(f(0) - f(1)) \neq 0 \quad \text{and} \quad \det f'(0) \neq 0 \, .
\]
Since $n \geq 3$, \tr{by Lemma~\ref{lem.GenericProjection} there exists
a matrix $A$ in $\M_{n, n-1}$ of maximal rank, 
such that $t \mapsto f(t) \cdot A$ defines an algebraic embedding
of $\CC$ into $\M_{n, n-1}$. Extend $A$ with an additional column 
$v \in \CC^n$ to a $n \times n$-matrix $(A | v)$ of determinant one. 
After applying the algebraic automorphism $X \to X \cdot (A | v)$ 
of $\SL_n$,} we can assume that the composition
\[
	\CC \stackrel{f}{\longrightarrow} \SL_n 
	\stackrel{\pi_{n-1}}{\longrightarrow} \tr{\M_{n,n-1}}
\] 
is an algebraic embedding.
After an algebraic coordinate change of $\SL_n$, we can assume
that the first column of $f(t)$ is equal to 
$(1, 0, \ldots, 0, t)^T$ by Lemma~\ref{lem:one}. Thus, up
to an algebraic automorphism of $\SL_n$, 
$f$ is the same as $t \mapsto E_{n1}(t)$ by Lemma~\ref{lem:two}.
This finishes the proof.
\end{proof}

\section{\texorpdfstring{Algebraic embeddings of $\CC$ into $\SL_2$}
{Algebraic embeddings of C into SL2}}

\begin{theorem}
	\label{thm.SL2}
	Any two algebraic embeddings $\CC \to \SL_2$ are the same 
	up to a holomorphic automorphism of $\SL_2$.
\end{theorem}

\begin{remark}
	\label{rem.reparametrization}
	Since for all $(a, b) \in \CC^\ast \times \CC$ the embeddings
	\[
		t \longmapsto  \begin{pmatrix}
			1 & t \\
			0 & 1
		\end{pmatrix}
		\qquad \text{and} \qquad
		t \longmapsto  \begin{pmatrix}
			1 & at+b \\
			0 & 1
		\end{pmatrix}
	\]
	are the same up to an algebraic automorphism 
	of $\SL_2$, it is enough to prove Theorem~\ref{thm.SL2} up
	to \tr{an algebraic} reparametrization of the embeddings $\CC \to \SL_2$.
\end{remark}

For the proof of Theorem~\ref{thm.SL2}, we need the following rather
technical result, which enables us to bring an arbitrary algebraic embedding 
$\CC \to \SL_2$ in a ``nice" position.

\begin{prop}
	\label{prop.help}
	Let $f \colon \CC \to \SL_2$ be an algebraic embedding.
	Then there exists a holomorphic automorphism $\varphi$ of $\SL_2$
	and a constant $a \in \CC$ such that
	the embedding 
	\[
		\tr{\CC \longrightarrow \SL_2 \, , \quad 
		t \longmapsto 
		\begin{pmatrix}
			g_{11}(t) & g_{12}(t) \\
			g_{21}(t) & g_{22}(t)
		\end{pmatrix}
		:= (\varphi \circ f)(t+a)}
	\]
	satisfies:
	\begin{enumerate}
		\item for all $t \in g_{11}^{-1}(0)$ we have $g_{12}(t) = t$;
		\item the map $t \mapsto (g_{11}(t), g_{21}(t))$
			is a proper, bimeromorphic immersion such that
			the image $\Gamma$ has only simple normal crossing singularities;
		\item the singularities of $\Gamma$ are distinguished by the first 
		coordinate of $\CC^2$;
		\item the line $\{ 0 \} \times \CC$ intersects $\Gamma$ transversally;
			\tr{in particular, $\Gamma$ is smooth in every point
			of $\Gamma \cap \{ 0 \} \times \CC$};
		\item the map $t \mapsto g_{11}(t)$ is polynomial.
	\end{enumerate}
\end{prop}

The proof of this proposition uses the following easy result which is a
direct application of the Baire category theorem:

\begin{lemma}
	\label{lem.interpolation}
	Let $\mathcal{H}(\CC^n)$ be the Fr\'echet space of holomorphic functions on 
	$\CC^n$ with the compact-open topology.
	If $S$ is the countable union of closed proper subspaces of 
	$\mathcal{H}(\CC^n)$, then $\mathcal{H}(\CC^n) \setminus S$
	is dense in \tr{$\mathcal{H}(\CC^n)$}.
\end{lemma}

Let $p \in \CC^n$ and let $i \in \{ 1, \ldots, n \}$. 
In our proof of Proposition~\ref{prop.help} we use the fact that
the linear functionals on $\mathcal{H}(\CC^n)$
\[
	h \longmapsto h(p) \quad \text{and} \quad h \longmapsto D_{x_i} h(p)
\]
are continuous and thus their kernels are proper closed subspaces 
\tr{of $\mathcal{H}(\CC^n)$}.

\tr{Additionally, we use for the proof of Proposition~\ref{prop.help} the
following, again rather technical result:

\begin{lemma}
	\label{lem.starting}
	Let $f \colon \CC \to \SL_2$ be an algebraic embedding.
	Then there exists an algebraic automorphism $\varphi$ of $\SL_2$
	such that the embedding 
	\[
		\CC \longrightarrow \SL_2 \, , \quad
		t \longmapsto (\varphi \circ f)(t) = 
		 \begin{pmatrix}
			x(t) & y(t) \\
			z(t) & w(t)
		\end{pmatrix}
	\]
	satisfies:
	\begin{enumerate}[a)]
		\item the maps $t \mapsto x(t)$ and
			$t \mapsto w(t)$ are non-constant polynomials;
		\item the maps $t \mapsto (x(t), z(t))$ and $t \mapsto (x(t), w(t))$
			 are bimeromorphic and immersive;
		\item the singularities of the image of $t \mapsto (x(t), z(t))$ lie
		         inside $(\CC^\ast)^2$;
		\item the image of $t \mapsto (x(t), z(t))$ intersects $\{ 0 \} \times \CC$
			transversally.
	\end{enumerate}	
\end{lemma}

\begin{proof}[Proof of Lemma~\ref{lem.starting}]
	Clearly, we can assume that $f(0)$ is the identity matrix $E_2 \in \SL_2$.
	By \cite[Theorem 4.14 and Remark 4.16]{ArFlKa2013Flexible-varieties},
	there exists an algebraic automorphism of $\SL_2$
	which fixes $E_2$ and maps the tangent vector 
	$f'(0) \in T_{E_2} \SL_2 =  \Lie \SL_2$ to the matrix 
	\[
		F_2 = \begin{pmatrix}
			1 & 0 \\
			0 & -1
		\end{pmatrix}
		\in \Lie \SL_2 \, .
	\]
	Thus we can assume that $f(0) = E_2$ and $f'(0) = F_2$.
	In particular, property a) is satisfied. Since $f'(t)$ is never zero
	and since $f'(t)$ is invertible for generic $t$ (note that $f'(0) = F_2$
	is invertible) it follows that $f'(t) \cdot v$ is non-zero for generic 
	$v \in \CC^2 \setminus \{ (0, 0) \}$. For generic $\mu \in \CC$, this implies 
	that the embedding
	\[
		t \longmapsto f(t) \cdot \begin{pmatrix}
					   1 & 0 \\
					   \mu & 1
				      \end{pmatrix}
	\]
	satisfies still property a) and the projection to the first column 
	gives an immersive map. Let us fix 
	such a $\mu$. For generic $\lambda \in \CC$
	the embedding
	\[
		t \longmapsto f(t) \cdot \begin{pmatrix}
					   1 & 0 \\
					   \mu & 1
				      \end{pmatrix}
				      \cdot \begin{pmatrix}
					   1 & \lambda \\
					   0 & 1
				      \end{pmatrix}
	\]
	still satisfies property a) and the projection to the first column 
	and the projection to the diagonal give immersive maps.
	Since any immersive morphism of $\CC$ to an irreducible affine 
	curve is birational, we can assume that $f$ satisfies properties a) and b).  
	Now, for generic $a \in \CC$ the embedding
	\begin{equation}
		\label{eq.first}
		\tag{\tr{$\ast\ast$}}
		t \longmapsto \begin{pmatrix}
					   1 & 0 \\
					   a & 1
				      \end{pmatrix}
				      \cdot f(t)
	\end{equation}
	satisfies properties a) and b)  and the singularities of the image 
	of the projection to the first column lie inside $\CC \times \CC^\ast$.
	Let us fix such an $a$.
	For generic $b \in \CC$ the embedding
	\begin{equation}
		\label{eq.second}
		\tag{\tr{$\ast\ast\ast$}}
		t \longmapsto \begin{pmatrix}
					   1 & b \\
					   0 & 1
				      \end{pmatrix}
				      \cdot
				      \begin{pmatrix}
					   1 & 0 \\
					   a & 1
				      \end{pmatrix}
				      \cdot f(t)		
	\end{equation}
	satisfies now the properties a) to c). Let $(p(t), q(t))^T$ be the first 
	column of the embedding~\eqref{eq.first}. Then the top left entry
	of the embedding~\eqref{eq.second} is given by $p(t) + b q(t)$.
	Now, if~\eqref{eq.second} satisfies properties a) to c),
	then~\eqref{eq.second} satisfies property d) 
	if and only if $p(t) + b q(t)$ has only simple roots. 
	However, this last
	condition is satisfied for generic $b$, since $p(t) + b q(t)$ has only simple
	roots if and only if for all $t$ the vector $(1, b)^T$ 
	lies not in the kernel of the matrix
	\[
		\begin{pmatrix}
			p(t) & q(t) \\
			p'(t) & q'(t)
		\end{pmatrix}
	\]
	and since this last matrix is invertible for generic $t$ and never vanishes.
	This finishes the proof.
\end{proof}
}

\begin{proof}[Proof of Proposition~\ref{prop.help}]
	\tr{Using Lemma~\ref{lem.starting} we can assume that
	$f$ satisfies the properties a) to d) of Lemma~\ref{lem.starting}.} 
	As a consequence of b) and c) we get that the map
	$t \mapsto (x(t), z(t), w(t))$ is a proper holomorphic embedding. 
	
	Let $t_1, \ldots, t_n$ be the roots of $x(t) = 0$ 
	\tr{(which are simple according to property d))}.
	After a reparametrization of $f$ of the form $t \mapsto t + a$ 
	one can assume that 
	$w(t_i) \neq w(t_j)$ for all $i \neq j$ and $t_i \neq 0$
	for all $i$.
	Let $a_i \in \CC$ such that $e^{-a_i} = - t_i z(t_i)$ and let 
	$b \colon \CC \to \CC$ be a polynomial map such that
	$b(w(t_i)) = a_i$ and $b'(w(t)) = 0$ for all $t$ with $x'(t) = 0$.
	After applying the holomorphic automorphism 
	\[
		\SL_2 \longrightarrow \SL_2 \, , \quad 
		\begin{pmatrix}
			x & y \\
			z & w
		\end{pmatrix} 
		\longmapsto
		\begin{pmatrix}
			x & e^{-b(w)}y \\
			e^{b(w)}z & w
		\end{pmatrix}
	\]
	we can assume that the embedding $f$ satisfies $y(t_i) = t_i$
	for all $i$ and $f$ still satisfies the properties a) to d).
	
	Let $\rho$ be the embedding $t \mapsto (x(t), y(t), z(t))$.
	Fix $x_0 \neq 0$ such that
	\begin{enumerate}[I)]
		\item $z(s) \neq 0$ and $x'(s) \neq 0$ for all $s \in x^{-1}(x_0)$ and
		\item the maps  $t \mapsto z(t)$ and $t \mapsto w(t)$ are 
			 injective on $x^{-1}(x_0)$.
	\end{enumerate}
	Let $\{ s_1, \ldots, s_n \} = x^{-1}(x_0)$.
	With the aid of Lemma~\ref{lem.interpolation} one can see that there exists
	a holomorphic function $c \colon \CC^2 \to \CC$ 
	that satisfies the following:
	\begin{enumerate}[i)]	
		\item for all $(x, z, w) \neq (x, z, w') \in \rho(\CC)$ we have
			$c(x, w) \neq c(x, w')$;
		\item for all $t$ with $x'(t) = 0$, the partial derivative
		         $D_w c$ vanishes in $(x(t), w(t))$;
		\item for all $i = 1, \ldots, n$ we have $c(0, w(t_i)) = 0$;
		\item for all integers $k, q$ and for all $2$-element sets
			$\{ i, j \} \neq \{ l, p \}$ we have
			\begin{align*}
				& \left[\log z(s_l) -\log z(s_p) +2\pi i q \right] \cdot
				\left[c(x_0, w(s_j)) - c(x_0, w(s_i))\right] \\
				& \neq 
				\left[\log z(s_i) -\log z(s_j)+2\pi i k\right] \cdot
				\left[c(x_0, w(s_p)) - c(x_0, w(s_l))\right] ;
			\end{align*}
		\item for all integers $k$ and for all $i \neq j$ we have
			\begin{align*}
				& \left[\log z(s_i) -\log z(s_j)+2\pi i k\right] \cdot
				\left[x'(s_i) c(x, w)'(s_j) - x'(s_j) c(x, w)'(s_i) \right] \\
				& \neq \left[ z'(s_i)x'(s_j) /z(s_i) -
				z'(s_j)x'(s_i) / z(s_j) \right] \cdot 
				\left[ c(x_0, w(s_j)) - c(x_0, w(s_i)) \right] \, .
		        \end{align*}
	\end{enumerate}
	Let $V \subseteq \CC^\ast$ be the largest subset such
	that for all $x_0 \in V$ the properties I) and II) are satisfied.
	By property a), the complement $\CC \setminus V$ is a closed
	discrete (countable) subset of $\CC$. The inequalities in
	iv) and v) are locally holomorphic in $x_0 \in V$ after a local
	choice of sections $s_1, \ldots, s_n$ of the covering
	$x^{-1}(V) \to V$ and a local choice of 
	the branches of the logarithms.
	Since $V$ is path-connected, one can now deduce
	that there exists a subset  
	$U \subseteq V$ such that $\CC \setminus U$
	is countable and for all $x_0 \in U$ the properties iv) and v) are satisfied.
	
	According to i) and c) there exists $\lambda \in \CC^\ast$
	such that for all $x_1 \in \CC \setminus U$ we have the following:
	If $(x_1, z, w) \neq (x_1, z', w') \in \rho(\CC)$, then 
	$e^{\lambda c(x_1, w)} z \neq e^{\lambda c(x_1, w')} z'$.
	Now, let $\varphi$ be the following holomorphic automorphism
	\[
		\SL_2 \longrightarrow \SL_2 \, , \quad
				\begin{pmatrix}
			x & y \\
			z & w
		\end{pmatrix} 
		\longmapsto
		\begin{pmatrix}
			x & e^{-\lambda c(x, w)}y \\
			e^{\lambda c(x, w)}z & w
		\end{pmatrix}
	\]
	and let $g = \tr{\varphi \circ f}$. 
	According to iii), $g$ satisfies property (1) of the proposition.
	Property ii) implies that \tr{$t \mapsto (g_{11}(t), g_{21}(t))$}
	is immersive. Clearly, \tr{$t \mapsto (g_{11}(t), g_{21}(t))$} is proper
	\tr{and $g$} satisfies property (5) of the proposition.
	By iii), it follows that $g$ satisfies property (4) 
	of the proposition and thus  \tr{$t \mapsto (g_{11}(t), g_{21}(t))$} 
	is bimeromorphic.
	By the choice of $\lambda$, it follows for $x_1 \notin U$
	that $g_{21}(t) \neq g_{21}(t')$ for all $t \neq t' \in x^{-1}(x_1)$.
	Since all $x_0 \in U$ satisfy iv) and v) 
	the image of $t \mapsto (g_{11}(t), g_{21}(t))$
	has only simple normal crossings, which have distinct first coordinates in 
	$\CC^2$. This implies properties (2) and (3) of the proposition.
\end{proof}

\begin{proof}[Proof of Theorem~\ref{thm.SL2}]
	Let $f \colon \CC \to \SL_2$ be an algebraic embedding. 
	We will prove that up to a holomorphic automorphism of $\SL_2$
	and up to \tr{an algebraic} reparametrization, $f$ is the same as
	the standard embedding $t \mapsto E_{12}(t)$.
	
	After applying
	a holomorphic automorphism of $\SL_2$ and performing
	\tr{an algebraic} reparametrization we can assume that
	$f$ satisfies properties (1) to (5) of Proposition~\ref{prop.help}.
	We denote 
	\[
		f(t) = \begin{pmatrix}
			x(t) & y(t) \\
			z(t) & w(t)
		\end{pmatrix} \, .
	\] 
	As usual,
	$\pi_1 \colon \SL_2 \to \CC^2$ denotes the projection
	onto the first column.
	Let $S$ be the (countable) closed discrete set of
	points $s \in \CC^2 \setminus \{ 0 \}$
	such that $(\pi_1 \circ f)^{-1}(s) = \{ s_1, s_2 \}$
	with $s_1 \neq s_2$, see property (2). For every $s$ in $S$,
	it holds that $y(s_1) \neq y(s_2)$, since $f$ is an embedding
	and since all simple normal crossings of
	the image of $\pi_1 \circ f$ lie inside $\CC^\ast \times \CC$
	due to property (4). Thus, we can choose $a_s \in \CC$ such that
	\[
		s_1 - e^{a_s} y(s_1) = s_2 - e^{a_s} y(s_2) \, .
	\]
	Let $\psi_1 \colon \CC \to \CC$ be a holomorphic
	function with $\psi_1(0) = 0$ such
	that for all $s \in S$ we have $\psi_1(x(s_1)) = a_s$.
	This function exists, since $x(s_1) = x(s_2) \neq 0$ for all $s \in S$ 
	(by property (4)),
	since $x((\pi_1 \circ f)^{-1}(S))$ is a closed analytic subset of $\CC$ 
	(by property (5)) 
	and since $x(s_1) \neq x(s'_1)$ for distinct $s$, $s'$ of $S$
	(by property (3)).
	Let $\alpha_1$ be the holomorphic automorphism of $\SL_2$
	defined by 
	\[
		\alpha_1 \begin{pmatrix}
			x & y \\
			z & w
		\end{pmatrix}
		=
	 	\begin{pmatrix}
			x & e^{\psi_1(x)} y \\
			e^{-\psi_1(x)} z & w
		\end{pmatrix}	\, .
	\]
	By composing $f$ with $\alpha_1$, we can assume that
	$s_1 - y(s_1) = s_2 - y(s_2)$ for all $s \in S$. The embedding
	$f$ still satisfies the properties (1) to (5).

	Let $\Gamma \subset \CC^2$
	be the image of $\pi_1 \circ f \colon \CC \to \CC^2$. By
	Remmert's proper mapping theorem
	\cite[Satz~23]{Re1957Holomorphe-und-mer},
	$\Gamma$ is a closed analytic subvariety of $\CC^2$.
	Now, using that $\pi_1 \circ f$ is immersive and
	$\Gamma$ has only simple normal crossings, 
	we get a holomorphic factorization
	\[	
		\xymatrix{
			\CC \ar[rd]_-{t \mapsto t-y(t)} \ar[r]^-{\pi_1 \circ f} &
			\Gamma \ar@{.>}[d]^-e \\
			& \CC \ar@{}[r]|<{\ \textrm{\normalsize.}} &
		}
	\]
	Using properties (1) and (4), it follows that the map
	\[
		\tilde{e} \colon \Gamma \longrightarrow \CC \, , \quad
		(x, z) \longmapsto \frac{e(x, z)}{x}	
	\]
	is holomorphic.
	Using Cartan's Theorem~B
	\cite[Th\'eor\`eme~B]{Ca1953Varietes-analytiqu}, we can
	extend $\tilde{e}$ to a holomorphic map $\psi_2 \colon \CC^2 \to \CC$.
	Let $\alpha_2$ be the holomorphic automorphism
	of $\SL_2$ defined by
	\[
		\alpha_2 \begin{pmatrix}
			x & y \\
			z & w
		\end{pmatrix}
		=
	 	\begin{pmatrix}
			x & y + x \psi_2(x, z) \\
			z & w + z \psi_2(x, z)
		\end{pmatrix}	\, .
	\]
	After applying the automorphism $\alpha_2$ to $f$ we can assume
	that $y(t) = t$. \tr{This implies that $x(0) w(0) = 1$. Let $p, q$ be 
	the holomorphic functions such that $p(t) t = x(t)- x(0)$ and
	$q(t) t = w(t)-w(0)$. After applying the holomorphic automorphism
	\[
		\begin{pmatrix}
			x & y \\
			z & w
		\end{pmatrix}
		\longmapsto
		\begin{pmatrix}
			w(0) & 0  \\
			0 & x(0)
		\end{pmatrix}		
		\begin{pmatrix}
			1 & 0  \\
			-q(y) & 1
		\end{pmatrix}
	 	\begin{pmatrix}
			x & y \\
			z & w
		\end{pmatrix}
		\begin{pmatrix}
			1 & 0 \\
			-p(y) & 1 
		\end{pmatrix}
	\]
	we can additionally assume that $w(t) = x(t) = 1$, which
	implies $z(t) = 0$. The statement follows now from 
	Remark~\ref{rem.reparametrization}.}
\end{proof}

\section{Acknowledgements}
The author would like to thank Peter Feller for many fruitful and stimulating
discussions. \tr{Many thanks go also to the referees for their helpful comments.}

\newcommand{\etalchar}[1]{$^{#1}$}
\providecommand{\bysame}{\leavevmode\hbox to3em{\hrulefill}\thinspace}
\providecommand{\MR}{\relax\ifhmode\unskip\space\fi MR }
\providecommand{\MRhref}[2]{%
  \href{http://www.ams.org/mathscinet-getitem?mr=#1}{#2}
}
\providecommand{\href}[2]{#2}

\end{document}